\documentclass[12pt]{amsart}
\usepackage{amscd}
\usepackage{verbatim}
\usepackage[russian,english]{babel}
\usepackage{amssymb, amsmath, amsthm, amscd,ifthen}
\usepackage[dvips]{graphics}
\usepackage[cp866]{inputenc}
\usepackage{graphicx}
\usepackage{epsfig}
\usepackage{tikz}
\usetikzlibrary{matrix,calc}

\usepackage{amsmath,amssymb,amscd,}
\usepackage[all]{xy}
\usepackage{url}

\usepackage[dvips]{graphics}
\usepackage[cp866]{inputenc}
\usepackage{graphicx}
\usepackage{epsfig}
\usepackage[hang]{subfigure}

\theoremstyle{plain}
\newtheorem{theorem}{Theorem}
\newtheorem{lemma}{Lemma}

\newtheorem{proposition}{Proposition}

\theoremstyle{definition}
\newtheorem{definition}{Definition}

\theoremstyle{plain}
\newtoks\thehProclaim
\newtheorem*{Proclaim}{\the\thehProclaim}

\theoremstyle{definition}
\newtoks{\thehRemark}
\newtheorem*{Remark}{\the\thehRemark}

\newtheorem{rrem}[definition]{Remark} 
\newtheorem{prop}[definition]{Proposition}  

\renewcommand{\leq}{\leqslant}
\renewcommand{\geq}{\geqslant}

\begin{document}

\title{ Ky Fan theorem for sphere bundles}

\author{ Gaiane Panina, Rade \v{Z}ivaljevi\'{c}}

\address{ G. Panina: St. Petersburg Department of Steklov Mathematical Institute; St. Petersburg State University;  gaiane-panina@rambler.ru;  }
\address{R. \v{Z}ivaljevi\'{c}: Mathematical Institute of the Serbian Academy of Sciences and Arts (SASA), Belgrade; rade@turing.mi.sanu.ac.rs }

\subjclass[2000]{}


\begin{abstract} The classic Ky Fan theorem is a combinatorial equivalent of Borsuk-Ulam theorem. It is a generalization and extension of Tucker's lemma and, just like its predecessor,  it pinpoints important properties of antipodal colorings of vertices of a triangulated sphere  $S^n$.
Here we  describe  generalizations of Ky Fan theorem for the case when the sphere is replaced by  the total space of a triangulated sphere bundle.
\end{abstract}

\maketitle

\section{Introduction}

Combinatorial statements, such as theorems of Carath\' eodory, Radon, Helly, Sperner, Tucker, Ky Fan, etc., are fundamental results of combinatorial (algebraic) topology,
 accessible to non-specialists, which are immediately applicable to mathematical economics, data science, game theory, graph theory, mathematical optimization,
computational geometry, and other fields.

\medskip
 Ky Fan theorem (Theorem \ref{ThmKyFan}) is also a combinatorial counterpart of the Borsuk-Ulam  theorem. It has many different proofs,  based on ideas  from both combinatorics and topology, see \cite{Alishahi-2015}  \cite{Kaiser} \cite{Mark} \cite{Meunier-Su} \cite{Musin2015} \cite{Musin2017} \cite{Rade-Combinatorica}.

\medskip For a general survey and a  discussion of the role of Tucker and Ky Fan theorem the reader is referred to \cite{Bull2019}, see also aforementioned references for  additional information and a guide to the older literature.

\medskip
There is still a relatively small number of pure topological statements  which admit an essentially  combinatorial reformulation or consequence.
Recall that such discrete analogues of theorems of Brouwer and Borsuk-Ulam are respectively Sperner and Tucker lemma \cite{Bull2019} \cite{mat08}. They both have a similar flavor, claiming the existence of particular patterns in special labelings of vertices of triangulated simplices/spheres.

\bigskip
In this paper we use \emph{parameterized Borsuk-Ulam theorem}, as developed in \cite{Jawor}, \cite{Nakaoka}, \cite{Dold}, to obtain qualitative and quantitative results about Tucker-Ky Fan labelings of triangulated sphere bundles.

\medskip
Recall the usual set-up of the classical Ky Fan theorem \cite{Ky Fan}.
The standard unit sphere $S^n \subset \mathbb{R}^{n+1}$ is triangulated and the triangulation is assumed to be centrally symmetric ($\mathbb{Z}_2$-invariant).
There is a labeling (coloring) of vertices of this triangulation
$$
\lambda:  Vert(S^n)\rightarrow \{\pm 1,...,\pm N \}
$$
which is \begin{itemize}
           \item \textit{antipodal},  $\lambda(-v)=-\lambda(v) \ \  \forall v\in Vert(S^n)$, and
           \item $\lambda(v) \neq -\lambda(w)$ for each pair $\{v,w\}$ of adjacent vertices of the triangulation.
         \end{itemize}

\medskip

The \textit{alternating number $Alt(\sigma)$  of a simplex $\sigma$  in }$S^n$  is the number of sign changes
in the labels of its vertices  $\lambda(Vert(\sigma))$, (where the labels are ordered by their absolute values). For example $Alt(-1,2,2,3,-4)=2;\ \ Alt(-1,2,-3,4)=3$, etc.

Clearly, the alternating numbers  of a simplex and its antipodal one are equal.
 The maximal possible alternating number is $n$, and these simplices come in pairs.

The Ky Fan theorem \cite{Ky Fan} states that if such a labeling is given then necessarily $n < N$, and the number of (pairs of)  simplices with alternating number $n$ must be odd:
\begin{theorem}\label{ThmKyFan}
In the above setting,
$$
\sum_{1\leq k_1< k_2...< k_{n+1}\leq N}
\alpha(k_1,-k_2, k_3,-k_4, . . . , (-1)^nk_{n+1}) \equiv 1 \, (mod \ 2),$$
where  $ \alpha(k_1,-k_2, k_3,-k_4, . . . , (-1)^nk_{n+1})$ is the number of $n$-simplices in $S^n$ whose vertices are {(bijectively)} colored by
 $ k_1,-k_2, k_3,-k_4, . . . , (-1)^nk_{n+1}$.
\end{theorem}

In the paper we address the following questions:

\textit{When is it possible to replace the triangulated sphere by some other triangulated manifold  with a free $\mathbb{Z}_2$-action? What happens if one replaces a unique sphere  $S^n$ by a parameterized  continuous family of spheres, that is, by the total space of some spherical  bundle over a smooth manifold?
}

\bigskip

Our main results are:

\begin{itemize}
  \item For a spherical bundle, there are ``many'' simplices  with alternating number $n$; taken together, they form a closed pseudomanifold which is topologically as complicated as the
base of the bundle, see Theorem \ref{Thm_0-eff}.
  \item
For non-trivial bundles one expects simplices with  alternating numbers bigger  than $n$. How much bigger depends on the  Stiefel-Whitney classes of the bundle,   see Theorem \ref{Thm_i-eff}.

\item Many explicit examples are provided in Section \ref{SectEx}. They include spherical bundles associated to the tangent bundles of
selected real and complex projective spaces.
\end{itemize}

\bigskip
\textbf{Acknowledgements. }It is our pleasure to acknowledge the  hospitality of the  Byurakan Astrophysical observatory, where this research was initiated, and also the  hospitality of Mathematisches Forschungsinstitut Oberwolfach, where in the
spring of 2024 this paper was completed as a  part of 'research in pairs' project. 
R. \v Zivaljevi\' c was supported by the Science Fund of the Republic of Serbia, Grant No.\ 7744592, Integrability and Extremal Problems in Mechanics, Geometry and Combinatorics - MEGIC.

\section{Notation and first observations}
Assume we have  a  fiber bundle
$$\pi:E\rightarrow B,$$
 where the base $B$ is a closed compact manifold of dimension $k$, each fiber is the sphere $S^n$,
and the total space $E$ is triangulated.

We assume, as before,  that there is an involution $inv:E\rightarrow E$ ($\mathbb{Z}_2$-action) without fixed points, which is simplicial (sends simplices linearly to simplices)   and keeps each fiber invariant. Assume that there is a labeling (coloring)
$$\lambda:  Vert(E)\rightarrow \{\pm 1,...,\pm N \}$$
which is antipodal with respect to the involution:
$$\lambda(v)=-\lambda(inv(v))\, .$$

We also assume  that  $\lambda(v) \neq -\lambda(w)$ for each pair $\{v,w\}$ of adjacent vertices of the triangulation.

\medskip

\textbf{Some conventions and abbreviations:}
{The triangulation and the coloring of the vertices, satisfying the above conditions, are simply called  \textit{nice colorings}.
We focus on spherical bundles which arise from vector bundles.\footnote{A vector bundle admits a metric and the unit spheres centered at the origins of the fibers constitute a spherical bundle.}   In this case we also speak of  a \textit{ nice coloring of a vector bundle  $E\rightarrow B$}, having in mind a nice coloring of the associated spherical bundle. }

\medskip
Let  $\diamondsuit^N := conv\{\pm e_1, . . . ,\pm e_N\}\subset \mathbb{R}^N$ be the $N$-dimensional \emph{cross-polytope}, whose vertices $\pm e_i$ are (signed versions of) standard basic vectors $e_i$ in $\mathbb{R}^N$. The central symmetry of the cross-polytope leads to a natural  antipodal action of
$\mathbb{Z}_2$ on its boundary $\partial \diamondsuit ^N$. We tacitly assume that the vertices of the cross-polytope are (tautologically) labeled   by  $\{\pm 1,...,\pm N \}$.
The \textit{alternating number $Alt(\sigma)$  of a simplex   in }$\partial \diamondsuit ^N$  is the number of sign changes
in the labels of its vertices (which are ordered by the absolute value of the corresponding labels).

A nice labeling $\lambda$ yields a simplicial antipodal map $$\Lambda: E \rightarrow  \partial \diamondsuit ^N. $$

The  \textit{alternating number}  of a simplex $\sigma$  in $E$  is  the alternating number of  its image $\Lambda(\sigma)$.

Since $\dim E=n+k$, the maximal possible alternating number is $n+k$.

\begin{rrem}
The sphere bundle $E$ is triangulated, however the fibers (homeomorphic to $S^n$) are not necessarily subcomplexes of $E$. As a consequence a labeling of $E$ does not descend
to a labeling of a fiber and a naive application of Ky Fan theorem, which would guarantee the existence of simplices with $Alt\geq n$, is not possible.

However, under a mild assumption that the projection map is  linear on each of the simplices of $E$, the reader can prove the existence of simplices with $Alt\geq n$ for every
nice coloring.\footnote{Assume there are no such simplices. Take any fiber, which is $S^n$. Intersecting $S^n$ with the simplices, we get an antipodal  decomposition of $S^n$.  Triangulate it equivariantly.  Each tile of the decomposition comes from some simplex and has a labeling. Using only these labels, put labels on the vertices of the resulted triangulation  (in an antipodal way). A contradiction with Ky Fan theorem.}
\end{rrem}

In the paper we are going to prove a much stronger result, see Theorem \ref{Thm_0-eff}.

\subsection*{Trivial versus non-trivial bundles. Preliminary examples}

We begin with simple examples showing that trivial and non-trivial bundles  behave quite differently.
\begin{proposition}\label{PropTrivial}
  A trivial $S^n$-bundle over a smooth base has a nice coloring in $2(n+1)$ colors $\{\pm 1,...,\pm (n+1) \}$. Therefore the maximal alternating number for a  coloring of a trivial bundle is $n$.
\end{proposition}
\begin{proof}
  Let $B\times S^n \rightarrow B$ be a trivial fiber bundle over a smooth manifold $B$. Fix a symmetric triangulation of $S^n$ with an equivariant coloring $\lambda$. For instance, one may take $S^n=\partial \diamondsuit^n$ together with the tautological coloring.
  Clearly, the maximal alternating number  for $(S^n, \lambda )$ is $n$.  Fix a triangulation of $B$. Taken together, the two triangulations give an equivariant regular cell decomposition $T$  of $B\times S^n$, where each cell is a product of two simplices. If we use $\lambda$ to color the vertices of $T$, we observe that no  two vertices of the same cell are colored by two opposite colors $i$ and $-i$. Since each cell is a product of two simplices, it can be triangulated  so that the arising triangulation of $T$ is equivariant. For example, the barycentric triangulation is symmetric.  There appear new vertices; color each of the vertices $v$ equivariantly and using the colors of the face of the cell decomposition containing $v$.  By construction, there are no edges in this triangulation  colored by two opposite colors $i$ and $-i$. Therefore, we obtain a nice coloring with $n+1$ different colors, so the maximal alternating number cannot be bigger than $n$.
\end{proof}

Let $B=S^2$, $n=1$. Let  $E_1$ be  the trivial bundle $S^2\times S^1$, and let $E_2$ be the Hopf bundle over $S^2$.

In both cases we necessarily have simplices  in $E_i$  with $Alt\geq 1.$
For the trivial bundle $E_1$, by Proposition \ref{PropTrivial}, there may not exist simplices with $Alt > 1.$

For the Hopf bundle $E_2$, the total space equals $S^3$, therefore, by Ky Fan theorem applied to $S^3$, there exists  a simplex with $Alt=3$.

\section{The main tool:  line bundle associated to a sphere bundle and powers of its first Stiefel-Whitney class}\label{SectMainTool}

A combinatorially oriented reader, not familiar with basics of the theory  of characteristic classes of vector bundles, may take these results for granted and informally follow their application. For introduction see initial sections of \cite{MilnorStashef} or \cite{Hatcher} or some other textbook on algebraic topology.

\subsection{Line bundle} For each simplicial complex $K$ with a free $\mathbb{Z}_2$-action, there exists an associated line bundle
\[
        \mathbb{R}^1 \longrightarrow K\times_{\mathbb{Z}_2} \mathbb{R}^1 \longrightarrow K/\mathbb{Z}_2
\]
with the corresponding first Stiefel-Whitney class $t := w_1(K)\in H^1(K/\mathbb{Z}_2; \mathbb{Z}_2)$.

\medskip The following result \cite{Husemoller} is of fundamental importance in the theory of characteristic classes of vector and sphere bundles.
Note that some authors use (\ref{eqn:relation}) as a convenient definition of Stiefel-Whitney classes.

\begin{proposition}\label{PropHomolGener} For a sphere fiber bundle $S^{n}\rightarrow E\longrightarrow B$, which arises as an associated bundle $E=S(V)$
of a rank $n+1$  vector bundle $V\longrightarrow B$,
  the cohomology $H^\ast(E/\mathbb{Z}_2)$ is described as a  $H^\ast(B)$-module generated by $1,t,t^2,...,t^n$, subject to the single  relation
\begin{equation}\label{eqn:relation}
  t^{n+1} + w_1t^n + w_2t^{n-1} +\dots + w_{n+1} = 0 \, ,
\end{equation} \noindent where $w_j = w_j(V)\in H^\ast(B)$ are the Stiefel-Whitney classes of the $(n+1)$-dimensional vector bundle $V$.

\medskip\noindent The single generator   $t$  is the Stiefel-Whitney class of the $\mathbb{Z}_2$-space $E = S(V)$.
\end{proposition}

\medskip
{\bf Example.}  If $n=1$ the relation (\ref{eqn:relation}) reduces to
\[
      t^2 = w_1t + w_2 \, .
\]
One can calculate inductively all other powers of $t$:
\[
     t^3 =t^2t = (w_1t + w_2)t = w_1(w_1t + w_2)+ w_2t = (w_1^2+w_2)t+w_1w_2 \,
\]
\[
     t^4 =t^2t^2 = (w_1t + w_2)^2 = w_1^2t^2 + w_2^2 = w_1^3t + w_1^2w_2 + w_2^2 \, , \mbox{ {\rm etc.} }
\]

\medskip
These are very special cases of general results which hold for all higher dimensional vector bundles with compact base.

\begin{proposition}\label{PropCheck}
    {  $$
  t^{n+j+1}=\sum_{i=1}^{n+1}  W_{i,j}(w_1,...,w_{n+1})t^{n-i+1} = W_{1,j}t^n+ W_{2,j}t^{n-1}+\dots + W_{n+1,j}
  $$ }
 where $W_{i,j}$ are polynomials in  Stiefel-Whitney classes. One can compute them inductively starting with $W_{i,0}=w_{i}$
  and proceeding by recursion
  {
  \begin{equation}\label{eq:recurrence} W_{i,j+1} = W_{i+1,j} +W_{1,j}w_{i} \end{equation}}
where $i=1,\dots, n+1$,  $j\geq 0$ and $W_{n+2,k}:= 0$ for each $k$.
\end{proposition}

{
The recurrence (\ref{eq:recurrence}) is a useful tool for concrete calculation. However, the following formula of Conner, see \cite[Section 2]{Con63} and  \cite[Section 3]{Con67}, is sometimes more convenient.
It  provides an explicit description of
$W_{i,j}$ in terms of the Stiefel-Whitney classes $w_i$ and the corresponding  dual classes $\bar{w}_i$.

\begin{proposition}\label{prop:Conn63}
The polynomial $W_{i,j}$ is homogenous of dimension $i+j$. It can be expressed as a truncated convolution of the total Stiefel-Whitney class $w = 1 + w_1+w_2 + \dots $ and the total dual class $\bar{w}=\frac{1}{w} = 1 + \bar{w}_1+\bar{w}_2 + \dots $ as follows
 \begin{equation}\label{eq:formula}
       W_{i,j}=  \sum_{k=0}^{j} \bar{w}_k w_{i+j-k}
\end{equation}
\end{proposition}

\subsection{Application}
We apply the theory outlined above to a triangulated spherical bundle with a nice coloring.
 Consider the line bundle
 \begin{equation}\label{eq:bundle1} E\times_{\mathbb{Z}_2}\mathbb{R}^1 \longrightarrow E/\mathbb{Z}_2.\end{equation}
  whose first  Stiefel-Whitney class we denote by $t_E$. It is the pullback $\Lambda^*$ of the bundle
\begin{equation}\label{eq:bundle} \partial\diamondsuit ^N\times_{\mathbb{Z}_2}\mathbb{R}^1 \longrightarrow \partial\diamondsuit ^N/\mathbb{Z}_2,\end{equation}
where $\Lambda$ is the simplicial map arising from the labeling. If $t_{\diamondsuit}$ is the first  Stiefel-Whitney class of the bundle (\ref{eq:bundle}), one concludes that for all $j$  $$t_E^j=\Lambda^*(t_{\diamondsuit}^j)  \, .$$

We pass to the Poincare duals of the classes $t_E^j$ and $t_{\diamondsuit}^j$, keeping in mind that in the dual language taking the pullback amounts essentially  to taking the preimage $\Lambda^{-1}$ of the dual homology class of $t_{\diamondsuit}^j$.
We make use of a very special representative of  the (dual of) class $t_{\diamondsuit}^j$.  The complex $\partial\diamondsuit ^N/\mathbb{Z}_2$  is the real projective space $\mathbb{R}P^{N-1}$ doubly covered by $S^{N-1},$ so the Poincare dual to $t_{\diamondsuit}$ is represented by (any) projective hyperplane. Therefore, the $j$-th power $t_{\diamondsuit}^j$ is represented by (any) projective  plane of codimension $j$.

\medskip
For our purposes we choose projective planes coming from  Euclidean  spaces  $$\mathbb{R}^N \supset e_0 \supset e_1 \supset e_2 \supset  \dots  $$ where
\begin{enumerate}
  \item { $codim\  e_j= j+1$.}
  \item $e_j$ is chosen to be in generic position w.r.t.\  all the faces of $\diamondsuit ^N$. That is the intersection of $e_j$ and $\sigma$ is empty, if $\dim (e_j)+ \dim (\sigma)< N$, or
  $$dim(e_j\cap \sigma)=  dim\ \sigma -j-1.$$
  \item $e_j$ intersects a face $\sigma$  of $\diamondsuit ^N$ iff  { $Alt(\sigma)\geq 1+j$.}
\end{enumerate}

\medskip {
The existence of such planes $e_0 \supset e_1 \supset e_2 \supset \dots$ is not difficult to establish.
For example $e_j$  can be defined as the kernel $Ker(M_j)$ of the linear map with the matrix
\[ M_j =
\begin{bmatrix}
    1 & 1 & 1 & \dots  & 1 \\
    x_1 & x_2 & x_3 & \dots  & x_N \\
    \vdots & \vdots & \vdots & \ddots & \vdots \\
    x_1^j & x_2^j & x_3^j & \dots  & x_N^j
\end{bmatrix}
\]
where $0<x_1< x_2 < \dots < x_N$   are any positive numbers. Indeed, let $v_i \, (i=1,\dots, N)$ be the column vectors of $M_j$.
Then $\lambda = (\lambda_i)\in Ker(M_j)\subset \mathbb{R}^N$ if and only if
$\sum_{i} \lambda_iv_i = 0$. (Such a vector $\lambda$ is called a \emph{dependency} of the collection $V = \{v_i\}$.)

\medskip
The set $I := \{i \mid \lambda_i\neq 0\}$ is called the support $supp(\lambda)$ of $\lambda$ and $sgn(\lambda) = \sum_{i} \lambda_ie_i$, where $\epsilon_i := sgn(\lambda_i)\in \{-1,0,+1\}$, is the corresponding \emph{sign pattern}.

\medskip
Let $I \subseteq [N]$ and let $\sigma = conv\{\eta_ie_i\}_{i\in I}$ be the corresponding face of $\diamondsuit ^N$, where $\eta_i\in \{-1,+1\}$. Let $\eta_j = 0$ for $j\notin I$. The vector
$sgn(\sigma):=\eta = (\eta_i)\in \mathbb{R}^N$ is called the \emph{sign pattern} of $\sigma$.

Then
\[
  e_j \cap {\stackrel \circ\sigma} = Ker(M_j) \cap {\stackrel \circ\sigma} \neq \emptyset
\]
where ${\stackrel \circ\sigma}$ is the relative interior of $\sigma$, if and only if there is a dependency $\lambda$ such that $sgn(\lambda) = sgn(\sigma)$.
This observation allows us to decide which faces of $\diamondsuit ^N$ are intersected by $e_j = Ker(M_j)$ simply by looking at the sign patterns of vectors in $e_j$.

\medskip
The following lemma is very useful for this purpose.
\begin{lemma}\label{lema:dep}
  Let $\{a_i\}_{i=1}^{d+1}$ be a collection of column vectors in $\mathbb{R}^d$ and let $A = [a_1, \dots, a_{d+1}]$ be the corresponding $d\times (d+1)$ matrix. Then
  \[
  \sum_{i=1}^{d+1} (-1)^i \det(A_i)a_i = 0
  \]
where $A_i$ is the square matrix obtained by removing the column $a_i$ from $A$.
\end{lemma}
By applying Lemma \ref{lema:dep} to the matrix $M_j$ (and its square submatrices) we are able to recover all dependences of the associated vector configuration.
In particular, we observe the characteristic alternating sign pattern for the dependences with the minimal support.

\medskip
The reader is referred to \cite[Section 9.4]{matroids} for more details about this construction, expressed in the language of the theory of oriented matroids.

\section{Generalizations of the Ky Fan theorem}

\subsection{Cycles $Z_{i}(E,\lambda)$   representing the powers $t_E^{n+i}$  of the Stiefel-Whitney class } For each $i=0,1,2,...,k-1$, where $k$ is the dimension of the base $B$,
the set
$$Z_i =\Lambda^{-1}(e_{n+i-1})$$ represents the (the Poincare dual of)  $t_E^{n+i}$.


\medskip

For a simplex $\sigma \in E$,  the set $Z_i(\sigma):=Z_i\cap \sigma$  is some (possibly empty) convex polytope. Indeed, the map $\Lambda$ is linear on $\sigma$. So {$Z_i(\sigma)=\Lambda ^{-1}(e_{n+i-1})\cap\sigma$ } is either some inner point of $\sigma$, or can be computed iteratively as the convex hull of $Z_i(\tau_j)$,
where $\tau_j$ ranges over the facets  of $\sigma$.

By construction, $Z_i$ intersects a simplex $\sigma \in E$ iff $Alt(\sigma )\geq n+i$.
In particular,  $Z_0$ captures the simplices with $Alt \geq n$.

\begin{definition}
 A simplex $\sigma$ in $E$ is \textit{ $i$-effective} if $Alt(\sigma)\geq n+i$.
\end{definition}
We are interested in $dim (Z_i(\sigma))$.
\begin{lemma} For a simplex $\sigma \in E$, we have
\begin{enumerate}
 \item $dim \ Z_i(\sigma)=dim (\sigma)-n-i$, if the simplex is $i$-effective.
  \item $Z_i(\sigma)=\emptyset$ otherwise.
\end{enumerate}
\end{lemma}

\begin{proof} We need to prove only (1).
It follows from transversality of the planes $e_i$ with respect to the faces of the cross polytope:

If $\sigma$ is an effective simplex which maps injectively by $\Lambda$ to some face of the cross polytope, $dim (Z_i(\sigma))= dim(\Lambda (\sigma)\cap e_{n+i-1})=dim (\sigma)-n-i$.

Otherwise $\Lambda$  factors through a projection map $$p:\sigma \rightarrow  \sigma'$$ to its $i$-effective face  for which $\Lambda$ is injective.

$$dim (Z_i(\sigma))= dim (Z_i(\sigma'))+ dim (\sigma)-dim (\sigma')= $$ $$(dim (Z_i(\sigma'))-dim (\sigma'))+ dim (\sigma)= dim (\sigma)-n-i
.$$

\end{proof}

\subsection{How many  $0$-effective simplices are there?}

Ky Fan theorem states that for every nice coloring there exist $0$-effective simplices in $E$, that is simplices with $Alt\geq n$.
 The advantage of using a sphere bundle $\pi: E \longrightarrow B$ (instead of a single sphere) is that we can provide more information about the
``size'' or complexity of the set of $0$-effective simplices.

\medskip
Let $pr_{n+i-1}:\mathbb{R}^N \rightarrow e_{n+i-1}^{\bot}$  be the  projection   to the orthogonal complement of $e_{n+i-1}$. Its kernel equals $e_{n+i-1}$.  Then  $Z_i  =  \Lambda ^{-1}(e_{n+i-1})= \Lambda ^{-1}\circ pr_{n+i-1}^{-1}(0)  \subset E.$   So $Z_i$ is \textit{the zero set }of the composition $ pr_{n+i-1}\circ \Lambda $.

\begin{theorem}\label{Thm_0-eff}  \begin{enumerate} Let $\Lambda: E \rightarrow  \partial \diamondsuit ^N$ be a nice coloring of an $n$-dimensional sphere bundle $\pi: E \rightarrow  B$. Then
                \item  $Z_0$ is a pure regular cell complex, a pseudomanifold of dimension $k$. Its cells are in bijection with $i$-effective simplices.
                \item The induced map
\[
    \pi^\ast : H^\ast(B) \longrightarrow H^{\ast}(Z_0)
\]
is a monomorphism.
                \item  The pushforward \newline $\pi_*: H_i(Z_0)\rightarrow H_i(B)$   is surjective.
                    \item The number of $n$-dimensional $0$-effective simplices is at least $k+1$.
 \end{enumerate}
\end{theorem}
\medskip
                Informally, (2) means that not only there are  many $0$-effective simplices, but also they   form
                something  which is topologically at least as complicated as the base is.

              \begin{proof}
               (1)    Maximal dimension of $Z_0(\sigma)$ is $k$, and it is attained on (some of) maximal simplices of $E$.
                Now take a $(k-1)$-effective simplex. It is a simplex of $E$ of codimension $1$, therefore, has exactly two
                adjacent maximal simplices in $E$. Both of them are $k$-effective.

 (2) follows  from \cite[Theorem 1.3]{Dold} (Corollary 1.5) which says that, cohomologically, the zero set $Z_0(E, \Lambda)$ is at least as complex as the base space $B$.

             (3) follows from (2).

             (4) follows also from (3). Indeed, $n$-dimensional $0$-effective simplices correspond to the vertices of $Z_0$. Since $dim Z_0=k$,
             it has at least $k+1$ vertices.
                \end{proof}

\textbf{Remark.} The bound (4) is clearly not sharp and leaves a lot of room for  improvements.

\subsection{Bundles  with    $i$-effective simplices.}


 Existence of $i$-effective simplices does not come for free, it depends on the height of the first Stiefel-Whitney class. The ``size'' of the zero set $Z_i$  (if non-empty)
 now gets  smaller as $i$ grows. More precisely, we have:

\begin{theorem}\label{Thm_i-eff} Let $\Lambda: E \rightarrow  \partial \diamondsuit ^N$ be a nice coloring of an $n$-dimensional sphere bundle $\pi: E \rightarrow  B$.  Assume in addition that the height of the first Stiefel-Whitney class of the bundle  $$E\times_{\mathbb{Z}_2} \mathbb{R}^1\rightarrow E/\mathbb{Z}_2$$ is at least $n+i$, that is,

$$ 0\neq t_E^{n+i} \in H^*(E/\mathbb{Z}_2, \mathbb{Z}_2).$$ Then \begin{enumerate}
                \item  $Z_i$ is a  non-empty cycle representing a nontrivial homology class in  $H_{k-i}((E/\mathbb{Z}_2), \mathbb{Z}_2)$.
                \item $Z_i$ is a pure $(k-i)$-dimensional regular cell complex, whose cells are in bijection with $i$-effective simplices.
                 \item The number of $n+i$-dimensional $i$-effective simplices is at least $k+1-i$.
              \end{enumerate}
\end{theorem}
              \begin{proof}
(1) $Z_i$ is non-trivial since it represents $t^i\neq 0$.
   (3) Indeed, $n+i$-dimensional $i$-effective simplices correspond to the vertices of $Z_i$. Since $\dim Z_i=k-i$,
             it has at least $k+1-i$ vertices.

                \end{proof}

\section{Concrete examples }\label{SectEx}

 We are going to use systematically the following principle:\textit{  if for a $S^n$-bundle $E\rightarrow B$ at least one of the classes $ W_{i,j}(w_1,...,w_{n+1})$  from Proposition \ref{PropCheck} is non-zero,
or, equivalently,  $t_E^{n+j+1}\neq 0$,  then  $N\geq n$,   and there are simplices with $Alt\geq n+j$.}

From here on we omit  $E$ from the subscript, setting $t:=t_E$.

\bigskip

\begin{enumerate}
  \item For a  nice coloring of an $S^n$-bundle with at least one non-zero  Stiefel-Whitney class $w_i$   ($i=1,2,...,n+1$), there are at least $k$ (pairs of)  $n+1$-dimensional simplices with $Alt\geq n+1$.
  \begin{proof}
    Relation (\ref{eqn:relation}) implies $t^{n+1}\neq 0$, and the statement follows by Theorem \ref{Thm_i-eff}.
  \end{proof}
  \item For a nice coloring of an oriented  $S^1$-bundles over  two-dimensional base $B=B^2$ with non-zero $w_2$   (that is, with an odd integer Euler class), there is always a simplex with $Alt=3$.

  \begin{proof} Proposition
   \ref{PropHomolGener}  implies that $t^3$ does not vanish since $w_1=0$, and therefore  $t^3=w_2t\neq 0$.
   \end{proof}

\noindent
   Question: What can we said about non-trivial oriented  $S^1$-bundles over  two-dimensional base $B=B^2$  (that is, with  $w_2=0$  but with an even integer Euler class)?
  \item For  a nice coloring of  an oriented $S^n$ bundle with  $w_2\neq 0$, there is always a simplex with $Alt\geq n+2$.
  \begin{proof}
    In this case $$t^{n+2}=(w_2+w_1^2)t^n+(w_3+w_1w_2)+...=w_2t^n+...\neq 0$$
  \end{proof}

  \item For  a nice coloring of an oriented $S^n$ bundle with  $w_3\neq 0$, there are at least $k-2$ simplices with $Alt\geq n+3$.

  \begin{proof}
     Indeed, in this case $$t^{n+3}=(w_3+w_1^3)t^n+(w_1w_3+w_2w_1^2w_2^2+w_4)t_E^{n-1}...= w_3t^n...\neq 0$$
  \end{proof}

\item Let $E$ be the total space of the   spherical bundle associated with the tangent bundle over $\mathbb{R}P^{4}$.  For any its nice triangulation there are simplices with $Alt=6$.

\begin{proof}
  It is known \cite{MilnorStashef}  that $w_1=a, \ w_4=a^4$, where $a$  is the generator of $H^1(\mathbb{R}P^{4})$, and the other classes vanish.
$$t^4=at^3+a^4,$$
$$t^5=a^2t^3+a^4t,$$
$$t^6=a^3t^3+a^4t^2,$$
but
$$t^7=a^4t^3+a^4t^3=0,$$
So there are at least $3$ simplices with $Alt=6$, but probably no simplices with  $Alt=7$.
\end{proof}

\item  For  any nice coloring of   the  tangent bundle of $\mathbb{R}P^{2^r }$,   there are at least three (pairs of)  simplices with $Alt\geq 2^{r+1}-2$.

\begin{proof}

Here we illustrate the efficiency of the formula from Proposition \ref{prop:Conn63}
by analysing the tangent bundle of $\mathbb{R}P^{n}$  in the case when  $n=2^r$ is an arbitrary power of $2$.

Following \cite[Section 4]{MilnorStashef}, the total Stiefel-Whitney class $w = w(\mathbb{R}P^{n})$ of the tangent bundle of $\mathbb{R}P^{n}$ is given by the formula

\[
   w(\mathbb{R}P^{n}) = (1+a)^{n+1} = \sum_{k=0}^{n} {{n+1}\choose{i}}a^n \, ,
\]
where $a\in H^1(\mathbb{R}P^{n})$ is the generator. If $n=2^r$ then $w = 1+a+a^{2^r}$ while the corresponding dual class is
\[
\bar{w} = 1 + a + a^2 +\dots + a^{2^r - 2} + a^{2^r-1}\, .
\]
An easy calculation shows that $W_{2,2^r-2}$ is non-zero. Indeed, in this case $i=2, j= {2^r-2}$ and the sum (\ref{eq:formula}) reduces to exactly one non-zero term corresponding to $\bar{w}_0 w_{2^r} = a^{2^r}$.

As an immediate consequence $t^{2^{r+1}-2}\neq 0$ and $Alt = 2^{r+1}-2$, which means that in this case the theoretical upper bound $2^r-1$ for the alternating number is almost attained.

\end{proof}

  \item For  any nice coloring of   the tangent bundle  of  $B=\mathbb{C}P^2$,   there are at least three  simplex with $Alt\geq 7$.

  Proof.  In our case $n=3$.

  Since $w_1=w_3=0$,
  we have  $t^4=w_2t^2+w_4$.
  One computes $$t^7=w_2^2t^3+w_4t^2$$  which is non-zero since $w_2^2 \neq 0$.

\end{enumerate}

\end{document}